\documentclass[11pt]{article} \usepackage{amssymb,amsmath,amsthm,graphicx} \usepackage{float} \usepackage{color,enumerate} \usepackage{fullpage} \usepackage{cite}

\newtheorem{theorem}{Theorem} \newtheorem{lemma}[theorem]{Lemma} \newtheorem{corollary}[theorem]{Corollary} \newtheorem{proposition}[theorem]{Proposition}   \newtheorem*{problem}{Open problem} 

\theoremstyle{definition} 

\newcommand{\po}{$1$-p.o.}

\title{Partial characterizations of $1$-perfectly orientable graphs} 

\author{ Tatiana Romina Hartinger\\ \small University of Primorska, UP IAM, Muzejski trg 2, SI6000 Koper, Slovenia\\ \small \texttt{tatiana.hartinger@iam.upr.si}\\ \and Martin Milani\v c\\ \small University of Primorska, UP IAM, Muzejski trg 2, SI6000 Koper, Slovenia\\ \small University of Primorska, UP FAMNIT, Glagolja\v ska 8, SI6000 Koper, Slovenia\\ \small \texttt{martin.milanic@upr.si} }

\usepackage{amssymb} \usepackage{graphicx}

\begin{document} \maketitle \begin{abstract} We study the class of $1$-perfectly orientable graphs, that is, graphs having  an orientation in which every out-neighborhood induces a tournament. $1$-perfectly orientable graphs form a common generalization of chordal graphs and circular arc graphs. Even though they can be recognized in polynomial time, little is known about their structure. In this paper, we develop several results on $1$-perfectly orientable graphs. In particular, we: (i) give a characterization of $1$-perfectly orientable graphs in terms of edge clique covers, (ii) identify several graph transformations preserving the class of $1$-perfectly orientable graphs, (iii) exhibit an infinite family of minimal forbidden induced minors for the class of $1$-perfectly orientable graphs, and (iv) characterize the class of $1$-perfectly orientable graphs within the classes of cographs and of co-bipartite graphs. The class of $1$-perfectly orientable co-bipartite graphs coincides with the class of co-bipartite circular arc graphs. \end{abstract}

\noindent\textbf{Keywords:} $1$-perfectly orientable graph, fraternally orientable graph, in-tournament digraph, structural characterization of families of graphs, cograph, co-bipartite graph, circular arc graph\\ \noindent\textbf{MSC (2010):} 05C20, 05C75, 05C05, 05C62, 05C69

\section{Introduction}

Many graph classes studied in the literature can be defined with the existence of orientations satisfying certain properties. In this paper we study graphs having an orientation that is an {\it out-tournament}, that is, a digraph in which in the out-neighborhood of every vertex induces an orientation of a complete graph. (An {\it in-tournament} is defined similarly.) Following the terminology of Kammer and Tholey~\cite{MR3152051}, we say that an orientation of a graph is {\it $1$-perfect} if the out-neighborhood of every vertex induces a tournament, and that a graph is {\it $1$-perfectly orientable} ({\it $1$-p.o.}~for short) if it has  a $1$-perfect orientation.\footnote{This notion is not to be confused with the notion of {\em perfectly orientable graphs}, defined as graphs that admit a linear vertex ordering (called a {\it perfect order}) $<$ such that $G$ contains no $P_4$ $a,b,c,d$ with $a < b$ and $d < c$.} In~\cite{MR3152051}, the authors introduced a hierarchy of graph classes, with \po~graphs being the smallest member of the family.
 Namely, they defined a graph to be {\it $k$-perfectly orientable} if in some orientation of it every out-neighborhood induces a disjoint union of at
 most $k$
 tournaments. In~\cite{MR3152051}, several approximation algorithms were given for optimization problems on
$k$-perfectly orientable graphs and related classes, and it is shown that recognizing $k$-perfectly orientable graphs is NP-complete for all $k\ge 3$. While the recognition complexity of $2$-perfectly orientable graphs seems to be unknown, \po~graphs can be recognized in polynomial time via a reduction to $2$-SAT~\cite[Theorem 5.1]{MR1244934}. A polynomial time algorithm for recognizing \po~graphs that works directly on the graph was  given by Urrutia and Gavril~\cite{MR1161986}.

The notion of \po~graphs was introduced in 1982 by Skrien~\cite{MR666799} (under the name $\{B_2\}$-graphs), where the problem of characterizing \po~graphs was posed. By definition, \po~graphs are exactly the graphs that admit an orientation that is an out-tournament. A simple arc reversal argument shows that that \po~graphs are exactly the graphs that admit an orientation that is an in-tournament. Such orientations were called {\it fraternal orientations} in several papers~\cite{MR1161986,MR1287025,MR1292980,MR1449722,MR1246675,MR2323998, MR2548660}. \po~graphs are also exactly the underlying graphs of the so-called locally in-semicomplete digraphs (which are defined similarly as in-tournaments, except that pairs of oppositely oriented arcs are allowed), see~\cite{MR2472389}.

While a structural understanding of 1-p.o.~graphs is still an open question, partial results are known. Bang-Jensen et al.~\cite{MR1244934} (see also~\cite{prisner1988familien}) proved a topological property of \po~graphs (stating that every \po~graph is $1$-homotopic), gave characterizations of \po~line graphs and of \hbox{\po}~triangle-free graphs, and showed that every graph representable as the intersection graph of connected subgraphs of unicyclic graphs is $1$-p.o. This implies that all chordal graphs and all circular arc graphs are \po, as observed already in~\cite{MR1161986} and in~\cite{MR1161986,MR666799}, respectively. Moreover, Bang-Jensen et al.~showed in~\cite{MR1244934} that every graph having a unique induced cycle of order at least $4$ is~\po, and conjectured that every \po~graph can be represented as the intersection graph of connected subgraphs of a cactus graph. The subclass of \po~graphs consisting of graphs that admit an orientation that is both an in-tournament and an out-tournament was characterized in~\cite{MR666799} (see also~\cite{MR1081957}) as exactly the class of proper circular arc graphs.

We now briefly discuss the characterizations of \po~graphs as intersection graphs and with forbidden substructures given by Urrutia and Gavril in~\cite{MR1161986}. The characterizations state that a graph is \po~if and only if it is the vertex-intersection graph of a family of mutually graftable subtrees in a graph (where two rooted subtrees of a graph are said to be {\it graftable} if their intersection is empty or contains one of their roots), if and only if it does not contain an induced {\it bicycle}, where a bicycle is a graph that can be written as a union of two monocycles satisfying a certain condition. A {\it monocycle} is a graph $G$ having a closed walk $C = (x_0,x_1,\ldots,x_k)$ with $x_k = x_0$, $k\ge 4$, and a walk $P = (y_1,\ldots, y_\ell)$ with $\ell \ge 2$ such that $V(G) = \{x_0,x_1,\ldots, x_k\}\cup \{y_1,\ldots, y_\ell\}$, for all $i\in \{1,\ldots, k\}$ (indices modulo $k$) vertex $x_i$ is neither equal nor adjacent to vertex $x_{i+2}$, and for all $j\in \{1,\ldots, \ell-2\}$ vertex $y_j$ is neither equal nor adjacent to vertex $y_{j+2}$; moreover, vertex $y_\ell$ is either adjacent to two consecutive vertices $s,t$ of $C$, or it is in $C$ and its immediate neighbours in $C$ are denoted $s,t$, and, finally, $y_{\ell-1}$ is distinct from and not adjacent to $s, t$. Unfortunately, the above characterization is not constructive and does not lead to an explicit list of minimal forbidden induced subgraphs for the class of \po~graphs. Bicycles can be recognized in polynomial time using the polynomial time recognition algorithm for \po~graphs, but are not structurally understood. It is also not clear how to detect an induced bicycle in a given non-$1$-perfectly orientable graph, except by running the algorithm by Urrutia and Gavril and  then applying the arguments given in the proof \hbox{of~\cite[Theorem 5]{MR1161986}}.

\newpage
In this paper, we prove several further results regarding $1$-perfectly orientable graphs.
 Our results can be summarized as follows: \begin{itemize}
  \item We give a characterization of \po~graphs in terms of edge clique covers similar to a known characterization of squared graphs due to
      Mukhopadhyay (Theorem~\ref{thm:characterization-edge-clique-covers}).
  \item We identify several graph transformations preserving the class of \po~graphs (Theorem~\ref{prop:operations}).
  In particular, we show that the class of \po~graphs is closed under taking induced minors. We also study the behavior of \po~graphs under the join
  operation (Theorem~\ref{prop:co-bipartite-join}), which  motivates the study of \po~co-bipartite graphs.
 \item We show that within the class of co-bipartite graphs, \po~graphs coincide with circular arc graphs
     (Theorem~\ref{thm:co-bipartite-circular-arc}). This adds to the list of the many characterizations of co-bipartite circular arc graphs.
  \item We identify several minimal forbidden induced minors for the class of \po~graphs, including
   ten small specific graphs and two infinite families: the complements of even cycles of length at least six and the complements of the graphs
   obtained from odd cycles by adding a component consisting of a single edge (Theorem~\ref{thm:non-1-po}).
\item Finally, we characterize \po~cographs, obtaining characterizations both in terms of forbidden induced subgraphs and
    in terms of structural properties (Theorem~\ref{thm:cograph}).
 \end{itemize}

The paper is structured as follows. In Section~\ref{sec:notation} we fix the notation. In Section~\ref{sec:edge-clique-covers} we give a characterization of \po~graphs in terms of edge clique covers. In Section~\ref{sec:operations}, we identify several graph transformations preserving the class of \po~graphs. In Section~\ref{sec:co-bipartite-graphs} we characterize \po~graphs within the class of complements of bipartite graphs.
Minimal forbidden induced minors for the class of \po~graphs are discussed in Section~\ref{sec:minors}. We conclude the paper with characterizations of \po~cographs in Section~\ref{sec:cographs}.

\section{Preliminaries}\label{sec:notation}

In this section, we recall the definitions of some of the most used notions in this paper. All graphs in this paper are simple and finite, but may be either directed (in which case we will refer to them as {\it digraphs}) or undirected (in which case we will refer to them as {\it graphs}). We use standard graph and digraph terminology. In particular, the vertex and edge sets of a graph $G$ will be denoted by $V(G)$ and $E(G)$, respectively, and the vertex and arc sets of a digraph $D$ will be denoted by $V(D)$ and $A(D)$. An edge in a graph connecting vertices $u$ and $v$ will be denoted by $\{u,v\}$ or simply $uv$. An arc in a graph connecting vertices $u$ and $v$ will be denoted by $(u,v)$. We will also use the notation $u\rightarrow v$ to denote the fact that an edge $uv$ of a graph $G$ is oriented from $u$ to $v$ in an orientation of $G$. The set of all vertices adjacent to a vertex $v$ in a graph $G$ will be denoted by $N_{G}(v)$, and its cardinality, the {\it degree of $v$ in $G$}, by $d_G(v)$. The {\it closed neighborhood} of $v$ in $G$ is the set $N_G(v)\cup \{v\}$, denoted by $N_G[v]$. An \emph{orientation} of a graph $G = (V, E)$ is a digraph $D= (V,A)$ obtained by assigning  a direction to each edge of $G$. A {\it tournament} is an orientation of a complete graph. Given a digraph $D$, the \emph{in-neighborhood} of a vertex $v$ in $D$, denoted by $N^{-}_{D}(v)$, is the set of all vertices $w$ such that $(w,v) \in A$. Similarly, the \emph{out-neighborhood} of $v$ in $D$ is the set of all vertices $w$ such that $(v,w) \in A$. The cardinalities of  the in- and the out-neighborhood of $v$ are the {\it in-degree} and the {\it out-degree of $v$} and are denoted by
 $d^{-}_{D}(v)$ and  $d^{+}_{D}(v)$, respectively.

Given two graphs $G$ and $H$, their {\it disjoint union} is the graph $G+H$ with vertex set $V(G)\cup V(H)$ (disjoint union) and edge set $E(G)\cup E(H)$. We write $2G$ for $G+G$. The {\it join} of two graphs $G$ and $H$ is the graph denoted by $G\ast H$ and obtained from the disjoint union of $G$ and $H$ by adding to it all edges joining vertex of $G$ with a vertex of $H$. Given a graph $G$ and a subset $S$ of its vertices, we denote by $G[S]$ the {\it subgraph of $G$ induced by $S$}, that is, the graph with vertex set $S$ and edge set $\{uv\in E(G)\mid u,v\in S\}$. By $G-S$ we denote the subgraph of $G$ induced by $V(G)\setminus S$, and when $S = \{v\}$ for a vertex $v$, we also write $G-v$.  The graph $G/e$ obtained from $G$ by \emph{contracting} an edge $e=uv$ is defined as $G/e=(V,E)$ where $V=(V(G)\setminus \{u,v\}) \cup \{w\}$ with $w$ a new vertex and $E=E(G-\{u,v\}) \cup \{wx \mid x\in N_{G}(u) \cap N_{G}(v)\}$. A {\it clique} (resp.,~{\it independent set}) in a graph $G$ is a set of pairwise adjacent (resp.,~non-adjacent) vertices of $G$. The {\it complement} of a graph $G$ is the graph $\overline{G}$ with the same vertex set as $G$ in which two distinct vertices are adjacent if and only if they are not adjacent in $G$. The fact that two graphs $G$ and $H$ are isomorphic to each other will be denoted by $G\cong H$. In this paper we will often not distinguish between isomorphic graphs. The path, the cycle, and complete graph on $n$ vertices will be denoted by $P_n$, $C_n$, and $K_n$, respectively, and the complete bipartite graph with parts of size $m$ and $n$ by $K_{m,n}$. Given two graphs $G$ and $H$, we say that $G$ is {\it $H$-free} if no induced subgraph of $G$ is isomorphic to $H$. For further background on graphs, we refer to~\cite{MR2744811}, on graph classes, to~\cite{MR2063679, MR1686154}, and on digraphs, to~\cite{MR2472389}.

A connected graph is said to be {\it unicyclic} if it has exactly one cycle.
The following consequence of~\cite[Corollary 5.7]{MR1244934} will be used in some of our proofs.

\begin{proposition}\label{prop:unicyclic} Every unicyclic graph is $1$-perfectly orientable.\end{proposition}

\section{A characterization in terms of edge clique covers}\label{sec:edge-clique-covers}

A graph $G$ is said to {\it have a square root} if there exists a graph $H$ with $V(H) = V(G)$ such that for all $u,v\in V(G)$, we have $uv\in E(G)$ if and only if the distance in $H$ between $u$ and $v$ is either $1$ or $2$. An {\it edge clique cover}  in a graph $G$ is a collection of cliques $\{C_1,\ldots, C_k\}$ in $G$ such that every edge of $G$ belongs to some clique $C_i$. In this section, we characterize \po~graphs in terms of edge clique covers, in a spirit similar to the well known  Mukhopadhyay's characterization of graphs admitting a square root, which we now recall.

\begin{theorem}[Mukhopadhyay~\cite{MR0210616}]\label{thm:squares} A graph $G$ with $V(G) = \{v_1,\ldots, v_n\}$ has a square root if and only if $G$ has an edge clique cover $\{C_1,\ldots, C_n\}$ such that the following two conditions hold: 
\begin{enumerate}[(a)]
  \item $v_i\in C_i$ for all $i$,
  \item for every edge $v_iv_j\in E(G)$, we have    $v_i\in C_j$ if and only if
 $v_j\in C_i$.
 \end{enumerate}
\end{theorem}

In the original statement of the theorem, the second condition is required for all $i\neq j$, but since $v_iv_j\not\in E(G)$ clearly implies $v_i\not \in C_j$ and $v_j\not \in C_i$, the equivalence in condition 2 automatically holds for all non-adjacent vertex pairs.

\begin{theorem}\label{thm:characterization-edge-clique-covers} For every graph $G$ with $V(G) = \{v_1,\ldots, v_n\}$, the following conditions are equivalent: \begin{enumerate}
  \item $G$ is $1$-perfectly orientable.
  \item $G$ has an edge clique cover $\{C_1,\ldots, C_n\}$
such that the following two conditions hold: \begin{enumerate}
  \item $v_i\in C_i$ for all $i$,
  \item for every edge $v_iv_j\in E(G)$, we have $v_i\in C_j$ or
 $v_j\in C_i$, but not both.
 \end{enumerate}
\item $G$ has an edge clique cover $\{C_1,\ldots, C_n\}$ such that the following two conditions hold: \begin{enumerate}
  \item $v_i\in C_i$ for all $i$,
  \item for every edge $v_iv_j\in E(G)$, we have $v_i\in C_j$ or
 $v_j\in C_i$.
 \end{enumerate}
 \end{enumerate}
\end{theorem}

Before proving Theorem~\ref{thm:characterization-edge-clique-covers}, we give two remarks. First, note that the difference between Theorem~\ref{thm:squares} and the equivalence of conditions 1 and 3 in Theorem~\ref{thm:characterization-edge-clique-covers} consists in replacing the equivalence in condition (b) of Theorem~\ref{thm:squares} with disjunction. This seemingly minor difference is in sharp contrast with the fact that recognizing graphs admitting a square root is NP-complete~\cite{MR1293386}, while \po~graphs can be recognized in polynomial time. Second, recall that a {\it pointed set} is a pair $(S,v)$ where $S$ is a nonempty set and $v\in S$. To every family ${\cal S} = \{(S_1,v_1),\ldots, (S_n,v_n)\}$ of pointed sets, one can associate a graph, the so called {\it catch graph} of ${\cal S}$ by setting $V(G) = \{v_1,\ldots, v_n\}$ and joining two distinct vertices $v_i$ and $v_j$ if and only if $v_i\in S_j$ or $v_j\in S_i$ (see, e.g.~\cite{MR1672910}). The equivalence between conditions 1 and 3 in the above theorem gives another proof of the fact that every \po~graph is the catch graph of a family of pointed sets, which also follows from the characterization of \po~graphs due to Urrutia and Gavril (stating that a graph is \po~if and only if it is the vertex-intersection graph of a family of mutually graftable subtrees in a graph)~\cite{MR1161986}.

\begin{proof}[Proof of Theorem~\ref{thm:characterization-edge-clique-covers}] First, we show the implication $1\Rightarrow 2$. Given a $1$-perfect orientation $D$ of a \po~graph $G$ with $V(G) = \{v_1,\ldots, v_n\}$, we define an edge clique cover $\{C_1,\ldots, C_n\}$ of $G$ by setting $C_i = \{v_i\}\cup N_D^+(v_i)$. By definition, each $C_i$ contains $v_i$, and, since $D$ is $1$-perfect, is a clique in $G$. Note that for all $i\neq j$, we have $v_j\in C_i$ if and only if $(v_i,v_j)\in A(D)$. In particular $v_j\in C_i$ and $v_i\in C_j$ cannot happen simultaneously. Since for every edge $v_iv_j\in E(G)$, we have either $(v_i,v_j)\in A(D)$ or $(v_j,v_i)\in A(D)$ but not both, condition 2(b) follows.

The implication $2\Rightarrow 3$ is trivial.

Finally, we show the implication $3\Rightarrow 1$. Suppose that $G$ has an edge clique cover $\{C_1,\ldots, C_n\}$ such that $v_i\in C_i$ for all $i$, and for every edge $v_iv_j\in E(G)$, $v_i\in C_j$ or $v_j\in C_i$. Define an orientation $D$ of $G$ as follows:
 for $1\le i<j\le n$ such that $v_iv_j\in E(G)$,
 set $(v_i,v_j)\in A(D)$ if $v_j\in C_i$, and
 $(v_j,v_i)\in A(D)$, otherwise.
 By definition, for every vertex $v_i\in V(G)$ we have
\begin{eqnarray*} 
 N_D^+(v_i) &=& \{v_j\mid j<i~\wedge~v_i\not\in C_j\}\cup \{v_j\mid j>i~\wedge~v_j\in C_i \}\\
 &  \subseteq & \{v_j\mid j<i~\wedge~v_j\in C_i\}\cup \{v_j\mid j>i~\wedge~v_j\in C_i\}~~\subseteq~~ C_i\,,
\end{eqnarray*} where the first inclusion relation holds due to condition 3(b).  Hence, $D$ is a $1$-perfect orientation of $G$, and $G$ is \po \end{proof}

For later use, we also record the following immediate consequences of Theorem~\ref{thm:characterization-edge-clique-covers}.

\begin{corollary}\label{cor:characterization-1po-complements} For every graph $G$ with $V(G) = \{v_1,\ldots, v_n\}$, the following conditions are equivalent: \begin{enumerate}
  \item $\overline{G}$ is $1$-perfectly orientable.
  \item $G$ has a collection of independent sets $\{I_1,\ldots, I_n\}$ such that the following two conditions hold:
\begin{enumerate}
  \item $v_i\in I_i$ for all $i$,
  \item for every non-adjacent vertex pair $v_iv_j\in E(\overline{G})$, we have $v_i\in I_j$ or
 $v_j\in I_i$.
\end{enumerate} \end{enumerate} \end{corollary}

\begin{corollary}\label{cor:upper-bound-on-theta} The edges of every $1$-perfectly orientable graph with $n$ vertices can be covered by $n$ cliques. \end{corollary}

Note that the converse of Corollary~\ref{cor:upper-bound-on-theta} does not hold. For example, the complement of the \hbox{$10$-vertex} graph $G_1$ (see Fig.~\ref{fig:1} on p.~\pageref{fig:1}) is not \po~(see Theorem~\ref{thm:non-1-po}), but can be edge-covered with (at most) $9$ cliques. 
(Determining if the edges of a given $n$-vertex graph can be covered by $n$ cliques is NP-complete~\cite{MR679638}; see also~\cite{MR712930}.)

\section{Operations preserving $1$-perfectly orientable graphs}\label{sec:operations}

In this section, we identify several operations preserving \po~graphs and characterize when the join of two graphs is $1$-p.o.
Two distinct vertices $u$ and $v$ in a graph $G$ are said to be \emph{true twins} if $N_{G}[u] = N_{G}[v]$, and \emph{false twins} if $N_{G}(u) = N_{G}(v)$. A vertex $v$ is \emph{simplicial} if its neighborhood forms a clique, and \emph{universal} if it is adjacent to every other vertex of the graph. The operations of adding a true twin, a simplicial vertex or a universal vertex to a given graph are defined in the obvious way. The operation of {\it duplicating a $2$-branch in the complement} of a graph $G$ is defined as follows. A {\it $2$-branch} in a graph $G$ is a path $(a,b,c)$ such that $d_G(b) = 2$ and $d_G(c) = 1$. We say that such a $2$-branch is {\it rooted at $a$}. {\it Duplicating a $2$-branch} $G$ results in a graph $H$ where $(a,b,c)$ is a $2$-branch in $G$, $V({H}) = V({G})\cup\{b',c'\}$, where $b'$ and $c'$ are new vertices, ${H}-\{b',c'\}= {G}$, and $(a,b',c')$ is a $2$-branch in $H$. Finally, the result of {\it duplicating a $2$-branch in the complement} of a graph $G$ is the complement of a graph obtained by duplicating a $2$-branch in $\overline{G}$.

\begin{theorem}\label{prop:operations} The class of $1$-perfectly orientable graphs is closed under each of the following operations: \begin{enumerate}
  \item Disjoint union.
  \item Adding a universal vertex (that is, join with $K_1$).
  \item Adding a true twin.
  \item Adding a simplicial vertex.
  \item Duplicating a $2$-branch in the complement.
  \item Vertex deletion.
  \item Edge contraction.
\end{enumerate} \end{theorem}

\begin{proof} For a \po~graph $G$, let us denote by $D(G)$ an arbitrary (but fixed) $1$-perfect orientation of $G$.

\medskip \noindent 1. If $G=G_1 +G_2$ is the disjoint union of two \po~graphs $G_1$ and $G_2$, then the disjoint union of $D(G_1)$ and $D(G_2)$ is a $1$-perfect orientation of $G$.

\medskip \noindent 2. Suppose we have a \po~graph $G$ with orientation $D(G)$ and we add a universal vertex $v$ to $G$, thus obtaining a graph $G'$. A $1$-perfect orientation $D'$ of $G'$ can be obtained by orienting an edge $xy\in E(G)$ from $x$ to $y$ if the edge is oriented from $x$ to $y$ in $D(G)$, and orienting the edges of the form $uv$ from $u$ to $v$. It is easy to see that $D'$ is indeed $1$-perfect.

\medskip \noindent 3. Let $w$ be a vertex in a \po~graph $G$, and let $G'$ be the graph obtained from $G$ by adding to it a true twin of $w$, say $v$. We obtain a $1$-perfect orientation $D'$ of $G'$ by maintaining the same orientation as in $D(G)$ for the edges in $G$ and orienting the new edges (incident with $v$) as $v\rightarrow{u}$ if $u \in N_{D(G)}^{+}(w)$, and $u\rightarrow{v}$ if $u \in N_{D(G)}^{-}(w)$. We also orient the edge between $w$ and $v$ as $w\rightarrow{v}$. It is a matter of routine verification to check that the so obtained orientation of $G'$ is $1$-perfect.

\medskip \noindent 4. If we add a simplicial vertex $v$ to a \po~graph $G$, then extending $D(G)$ by orienting all edges incident with $v$ away from $v$ results in an orientation $D'$ of the new graph, say $G'$, such that $N_{D'}^{+}(v)$ is a clique in $G'$. The other out-neighborhoods were not changed, so they are cliques in $G'$ as well.

\medskip \noindent 5. Let $V(G) = \{v_1,\ldots, v_n\}$. If $G$ is \po, then Corollary~\ref{cor:characterization-1po-complements} applies to $\overline{G}$. Hence, $\overline G$ has a collection of independent sets $\{I_1,\ldots, I_n\}$ such that $v_i\in I_i$ for all $i$, and for every edge $v_iv_j\in E(G)$, we have $v_i\in I_j$ or
 $v_j\in I_i$. Let $H$ be the graph resulting from duplicating a $2$-branch
 $(a,b,c)$ in $\overline{G}$; without loss of generality, we may assume that
 $(a,b,c) = (v_1,v_2,v_3)$; furthermore, let the two new vertices $b'$ and $c'$ be labeled as
 $v_{n+1}$ and $v_{n+2}$, respectively.
 It suffices to prove that
 $H$ has a collection of independent sets $\{J_1,\ldots, J_{n+2}\}$
such that $v_k\in J_k$ for all $k$, and for every edge $v_iv_k\in E(\overline{H})$, we have $v_i\in J_k$ or
 $v_k\in J_i$.
 We may assume without loss of generality that the sets $I_j$ are maximal independent sets in $\overline{G}$,
 which in particular implies that each $I_j$ contains exactly one of the vertices $b$ and $c$.
 We define the sets $J_k$ for $k\in \{1,\ldots, n+2\}$ with the following rule:
 \begin{itemize}
   \item For all $v_k\in V(G)$, set $$J_k = \left\{
                                                      \begin{array}{ll}
                                                        I_k\cup\{b'\}, & \hbox{if $b\in I_k$;} \\
                                                        I_k\cup \{c'\}, & \hbox{if $c\in I_k$.}
                                                      \end{array}
                                                    \right.$$
   \item For $k = n+1$ (that is, $v_k = b'$), set $J_k = (I_2\setminus \{b\})\cup\{b',c\}$.
   \item For $k = n+2$ (that is, $v_k = c'$), set $J_k = (I_3\setminus \{a,c\})\cup\{b,c'\}$.
 \end{itemize}
Clearly, each $J_k$ is an independent set in $H$. Let $v_iv_k\in E(\overline{H})$. Since $b'c'\not\in E(\overline{H})$, we may assume that $v_i\in V(G)$. We analyze three cases according to where is $v_k$.

If $v_k\in V(G)$, then $v_iv_k\in E(G)$ and hence $v_i\in I_k$ or $v_k\in I_i$, implying $v_i\in J_k$ or $v_k\in J_i$.

If $v_k= b'$, then either $v_i\in J_k$ (in which case we are done), or $v_i\not \in J_k = (I_2\setminus \{b\})\cup\{b',c\}$, in which case either $v_i = b$ or $v_i\not\in I_2$. In the former case, we have $i = 2$ and $v_k = b'\in J_2$, while in the latter case, we have $b = v_2\in I_i$, which implies $v_k = b'\in J_i$.

If $v_k= c'$, then either $v_i\in J_k$ (in which case we are done), or $v_i\not \in J_k = (I_3\setminus \{a,c\})\cup\{b,c'\}$, in which case either $v_i \in \{a,c\}$ or $v_i\not\in I_3$. In the former case, we have $c\in I_i$ (if $v_i = a$ this follows from the maximality of $I_i$) and consequently $v_k = c'\in J_i$. In the latter case, we have $c = v_3\in I_i$, which implies $v_k = c'\in J_i$.

We have shown that $v_k\in J_k$ for all $k$, and for every edge $v_iv_k\in E(\overline{H})$, we have $v_i\in J_k$ or
 $v_k\in J_i$. By Corollary~\ref{cor:characterization-1po-complements}, $H$ is the complement of a \po~graph,
which establishes item 5.

\medskip \noindent 6. Closure under vertex deletions follows immediately from the fact that the class of complete graphs is closed under vertex deletions.

\medskip \noindent 7. Let $e = uv$ be an edge of a \po~graph $G$, and let $D$ be a $1$-perfect orientation of $G$, with (without loss of generality) $u\rightarrow v$. Let $G'=G/e$ be the graph obtained by contracting the edge $e$, and let $w$ be the vertex replacing $u$ and~$v$.

Set 
\begin{eqnarray*}
 X &=& N_G(u)\setminus N_G(v)\,,\\
  Y &=& \{x\in N_G(u)\cap N_G(v)\mid (x,v) \in A(D)\}\,,\\
  U &=& \{x\in N_G(u)\cap N_G(v)\mid (v,x) \in A(D)\}\,,\\
  W &=& \{x\in N_G(v)\setminus N_G(u)\mid (x,v) \in A(D)\}\,,\\
  Z &=& \{x\in N_G(v)\setminus N_G(u)\mid (v,x) \in A(D)\}\,,\\
R &=& V(G)\setminus(X\cup Y\cup U \cup W\cup Z\cup \{u,v\})\,.
\end{eqnarray*}

Let $D'$ be an orientation of $G'$ defined as follows: \begin{enumerate}[(i)]
  \item For all edges $e \in E(G')$ whose endpoints are not incident with $w$, orient $e$ the same way as it is oriented in $D$.
  \item For all $x \in  X$, orient the edge $xw$ as $x \rightarrow{w}$.
  \item For all $x \in Y$, orient the edge $xw$ as $x \rightarrow{w}$.
  \item For all $x \in U$, orient the edge $xw$ as $w \rightarrow{x}$.
  \item For all $x \in W$, orient the edge $xw$ as $x \rightarrow{w}$.
  \item For all $x \in Z$, orient the edge $xw$ as $w \rightarrow{x}$.
\end{enumerate}

We complete the proof by showing that $D'$ is a $1$-perfect orientation of $G'$. We do this by directly verifying the defining condition that for every vertex $x$ of $V(G')$, the set $N_{D'}^+(x)$ is a clique in $G'$. Note that $X\cup Y\cup U\cup W\cup Z\cup \{w\}\cup R$ is a partition of $V(G')$. We consider seven cases depending on to which part of this partition $x$ belongs.

\medskip \noindent (1) $x \in X$. In this case, $N_{D'}^{+}(x)=(N_{D}^{+}(x)\setminus \{u\}) \cup \{w\}$. Note that since $(u,v)\in A(D)$ and $D$ is a $1$-perfect orientation of $G$, we have $u\in N^+_D(x)$. Consequently, since $N^+_D(x)$ is a clique in $G$ containing $u$, it contains no vertex from $R\cup Z$, and thus $N_{D'}^{+}(x)=(N_{D}^{+}(x)\setminus \{u\}) \cup \{w\}$ is a clique in $G'$.

\medskip \noindent (2) $x \in W$. In this case, $v \in N_{D}^{+}(x)$, and a similar reasoning as above shows that $N_{D'}^{+}(x)=(N_{D}^{+}(x) \setminus \{v\}) \cup \{w\}$ is a clique in $G'$.

\medskip \noindent (3) $x \in Z$. In this case, $N_{D'}^{+}(x)=N_{D}^{+}(x)$ and this set is a clique in $G$ and hence in $G'$.

\medskip \noindent (4) $x \in Y$. In this case, we have two possibilities, either $u \in N_{D}^{+}(x)$ or not. In the former case, we have $N_{D'}^{+}(x)=(N_{D}^{+}(x) \setminus \{u,v\}) \cup \{w\}$ which is a clique in $G'$, since $N_{D}^{+}(x) $ is a clique in $G$ containing $u$ and $v$, and every neighbor of $w$ in $G'$ is a neighbor of either $u$ or of $v$ in $G$. In the latter case, we have $N_{D'}^{+}(x)=(N_{D}^{+}(x) \setminus \{v\}) \cup \{w\}$, which is again a clique in $G'$ by a similar argument.

\medskip \noindent (5) $x\in U$.
 Now, $N_{D'}^{+}(x)=N_{D}^{+}(x) \setminus \{u\}$, which is a clique in $G$ not containing $u$ or $v$, and hence a clique in $G'$.

\medskip \noindent (6) $x \in R$. Since the edges with endpoints in $R$ have no endpoint in $\{u,v\}$, the edges which have $x$ as an endpoint will maintain the same orientation as in $D$. Therefore, $N_{D'}^{+}(x)=N_{D}^{+}(x)$ is a clique in $G'$.

\medskip \noindent (7) $x=w$. In this case, we have $N_{D'}^{+}(x)= N_{D}^{+}(v)$, therefore $N_{D'}^{+}(x)$ forms a clique in $G'$.
\end{proof}

In the study of \po~graphs we may restrict our attention to connected graphs. It is a natural question whether we may also assume that $G$ is co-connected, that is, that its complement is connected, or, equivalently, that $G$ is not the join of two smaller graphs. The join operation does not generally preserve the class of \po~graphs: the graphs $2K_1$ and $3K_1$ are trivially \po, but their join, $K_{2,3}$, is not (as can be easily verified; see also Theorem~\ref{thm:non-1-po}). In the next theorem we characterize when the join of two graphs is $1$-p.o. A graph is said to be {\it co-bipartite} of its complement is bipartite.

\begin{theorem}\label{prop:co-bipartite-join}
Suppose that a graph $G$ is the join of two graphs $G_1$ and $G_2$. Then, $G$ is $1$-perfectly orientable if and only if one of the following conditions hold: \begin{enumerate}
  \item $G_1$ is complete and $G_2$ is \po, or vice versa.
  \item Each of $G_1$ and $G_2$ is a co-bipartite \po~graph.
\end{enumerate} In particular, the class of co-bipartite \po~graphs is closed under join. \end{theorem}

\begin{proof} Suppose first that $G$ is $1$-p.o. Clearly, both $G_1$ and $G_2$ are \po~graphs. If one of $G_1$ or $G_2$ is complete or both are co-bipartite, we are done. So suppose that neither of them is complete and $G_1$, say, is not co-bipartite. Then, $G_1$ contains the complement of an odd cycle, $\overline{C_{2k+1}}$ for some $k\ge 1$, as induced subgraph. Since $G_2$ is not complete, it contains $2K_1$ as induced subgraph. Consequently, $G$ contains the join of $\overline{C_{2k+1}}$ and $2K_1$ as induced subgraph. As this graph is isomorphic to the complement of $C_{2k+1}+K_2$, it is not \po~(see Theorem~\ref{thm:non-1-po}), and hence neither is $G$, a contradiction.

For the converse direction, suppose first that $G_1$ is complete and $G_2$ is \po, or vice versa. In this case $G$ is~\po, since it can be obtained from a \po~graph by a sequence of universal vertex additions, and Theorem~\ref{prop:operations} applies. Suppose now that $G_1$ and $G_2$ are two co-bipartite \po~graphs with bipartitions of their respective vertex sets into cliques $\{A_1, B_1\}$ and $\{A_2,B_2\}$, respectively (one of the two cliques in each graph can be empty). Fixing a $1$-perfect orientation $D_i$ of each $G_i$ (for $i = 1,2$), we can construct a $1$-perfect orientation, say $D$,  of $G = G_1\ast G_2$, as follows. Every edge of $G$ that is an edge of some $G_i$ is oriented as in $D_i$. Orient the remaining edges of the join from $A_1$ to $A_2$, from $B_1$ to $B_2$, from $A_2$ to $B_1$ and from $B_2$ to $A_1$.  Let us verify that the out-neighborhood of a vertex $x\in A_1$ with respect to $D$ forms a clique in $G$ (the other cases follow by symmetry). We have $N_D^+(x) = N_{D_1}^+(x)\cup A_2$, and since $N_{D_1}^+(x)$ is a clique in $G_1$, $A_2$ is a clique in $G$ and there are all edges between $G_1$ and $A_2$, the set $N_D^+(x)$ is indeed a clique in $G$. This shows that $G$ is \po{}

Since the class of bipartite graphs is closed under disjoint union, the class of co-bipartite graphs is closed under join. Consequently, the set of co-bipartite \po~graphs is closed under join. \end{proof}

\section{$1$-perfectly orientable co-bipartite graphs}\label{sec:co-bipartite-graphs}

The behavior of \po~graphs under the join operation motivates the study of \po~co-bipartite graphs. In this section we show that a co-bipartite graph is \po~if and only if it is circular arc. A graph is {\it circular arc} if it is the intersection graph of a set of closed arcs on a circle. The class of circular arc graphs forms an important and well studied subclass of \po~graphs; see, e.g.,~\cite{MR3159129, MR2567965}. 
The equivalence of the classes of \po~graphs and circular arc graphs within co-bipartite graphs will be derived using two ingredients: a necessary condition for the \po~property, which holds in general, and a characterization of co-bipartite circular arc graphs due to Hell and Huang.

We say that a chordless cycle $C$ in a graph $G$ is oriented {\it cyclically} in an orientation $D$ of $G$ if every vertex of the cycle has exactly one out-neighbor on the cycle (see~\cite{MR2370526,MR2643278} for results on orientations defined in terms of this property).

\begin{lemma}\label{lem:cycles} In every $1$-perfect orientation $D$ of a \po~graph $G$, every chordless cycle of length at least four is oriented cyclically. \end{lemma}

\begin{proof} Suppose that a chordless cycle $C$ in $G$ is not oriented cyclically in some $1$-perfect orientation $D$ of $G$. Let $C'$ be the orientation of $C$ induced by $D$. By assumption, $C$ contains a vertex $v$ with $d^+_{C'}(v) \neq 1$. Since $\sum_{u\in V(C)}d^+_{C'}(u) = |A(C')| = |E(C)| = |V(C)|$, it is not possible that $d^+_{C'}(u) \le 1$ for all $u\in V(C)$, as this would imply $d^+_{C'}(v) = 0$ and consequently $\sum_{u\in V(C)}d^+_{C'}(u) <|V(C)|$. Thus, $C$ contains a vertex $v$ with $d^+_{C'}(v) = 2$. Since $C$ is of length at least $4$ and chordless, the out-neighborhood of $v$ in $C'$, and hence in $D$, is not a clique in $G$, contradicting the fact that $D$ is a $1$-perfect orientation of $G$. \end{proof}

We now describe the characterization of co-bipartite circular arc graphs due to Hell and Huang. Let $G$ be a co-bipartite graph with a bipartition $\{U, U'\}$ of its vertex set into two cliques. An edge of $G$ connecting a vertex from $U$ with a vertex of $U'$ is said to be a \emph{crossing edge} of $G$. A coloring of the crossing edges of $G$ with colors red and blue is said to be \emph{good} (with respect to $\{U, U'\}$) if for every induced $4$-cycle in $G$, the two crossing edges in it are of the opposite color. The following characterization of co-bipartite circular arc graphs is a reformulation of~\cite[Corollary 2.3]{MR1435657}.

\begin{theorem}\label{thm:co-bip-circ-arc} Let $G$ be a co-bipartite graph with a bipartition $\{U, U'\}$ of its vertex set into two cliques. Then $G$ is a circular arc graph if and only if it has a good coloring with respect to $\{U, U'\}$. \end{theorem}

\begin{theorem}\label{thm:co-bipartite-circular-arc} For every co-bipartite $G$ graph, the following properties are equivalent: \begin{enumerate}
 \item $G$ is $1$-perfectly orientable. 
 \item $G$ has an orientation in which every induced $4$-cycle is oriented cyclically. 
 \item $G$ is circular arc. \end{enumerate} \end{theorem}

\begin{proof} 
As shown by Skrien~\cite{MR666799}, implication $3\Rightarrow 1$ holds for general (not necessarily co-bipartite) graphs. 
Similarly, the implication $1\Rightarrow 2$ holds in general as follows from Lemma~\ref{lem:cycles}.

It remains to prove that if $G$ is co-bipartite, then condition 2 implies condition 3. Let $D$ be an orientation of $G$ in which every induced $4$-cycle of $G$ is oriented cyclically. Fix a partition $\{U,U'\}$ of $V(G)$ into two cliques. We will now show that $G$ admits a good coloring (with respect to $\{U,U'\}$), and Theorem~\ref{thm:co-bip-circ-arc} will imply that $G$ is circular arc. We obtain a good coloring of $G$ as follows: for every crossing edge $e$ of $G$, we color $e$ red if the arc of $D$ corresponding to $e$ goes from $U$ to $U'$, and blue if it goes from $U'$ to $U$. To see that this is indeed a good coloring, let $C$ be an arbitrary induced $4$-cycle of $G$. Since $C$ is oriented cyclically in $D$, out of the two crossing edges of $C$ exactly one is oriented from $U$ to $U'$ in $D$. This implies that the two crossing edges of $C$ will have different colors in the above coloring. It follows that the obtained coloring is a good coloring, as claimed. \end{proof}

Many characterizations of circular arc co-bipartite graphs are known, including a characterization in terms of forbidden induced subgraphs due to Trotter and Moore~\cite{MR0450140} and several (at least five) others, see, e.g.,~\cite{MR2567965,MR3159129}. By Theorem~\ref{thm:co-bipartite-circular-arc}, each of these yields a characterization of \po~co-bipartite graphs. Theorem~\ref{thm:co-bipartite-circular-arc} can also be seen as providing further characterizations of co-bipartite circular arc graphs.

\section{A family of minimal forbidden induced minors}\label{sec:minors}

Theorem~\ref{prop:operations} implies that the class of \po~graphs is closed under vertex deletions and edge contractions. Hence, it is also closed under taking induced minors. Recall that a graph $H$ is said to be an {\it induced minor} of a graph $G$ if $H$ can be obtained from $G$ by a series of vertex deletions or edge contractions. Graph classes closed under induced minors include all the minor-closed graph classes, as well as many others (see, e.g.,~\cite{MR1415290,MR2901082, MR1360111, MR1109419, MR0276129}). Since the class of \po~graphs is closed under induced minors, it can be characterized in terms of {\it minimal forbidden induced minors}.
That is, there exists a unique minimal set of graphs $\tilde{\cal F}$ such that (i) a graph $G$ is \po~if and only if $G$ is {\it $\tilde{\cal F}$-induced-minor-free} (that is, no induced minor of $G$ is isomorphic to a member of $\tilde{\cal F}$), and (ii) every proper induced minor of every graph in $\tilde{\cal F}$ is \po{} In this section we identify an infinite subfamily ${\cal F}\subseteq \tilde{\cal F}$ of minimal forbidden induced minors for the class of \po~graphs.

\begin{sloppypar} We start with two preliminary observations. The fact that every circular arc graph \hbox{is~\po} implies the following. \end{sloppypar}

\begin{proposition}\label{prop:powers-of-cycles} The complement of every odd cycle is~$1$-perfectly orientable.\end{proposition}

\begin{proof} Recall that the {\it $k$-th power} of a graph $G$ is the graph with the same vertex set as $G$ in which two distinct vertices are adjacent if and only if their distance in $G$ is at most $k$. It is easy to see (and also follows from the fact that the class of circular arc graphs is closed under taking powers~\cite{MR1206558}) that all powers of cycles are circular arc graphs. Therefore, the fact that the complement of every odd cycle is \po~follows from two facts: (i) that the complement of $C_{3}$ is~\po, and (ii) for every $k\ge 2$, the complement of the odd cycle $C_{2k+1}$ is isomorphic to a power of a cycle, namely to $C_{2k+1}^{k-1}$. \end{proof}

Since every disjoint union of paths is an induced subgraph of a sufficiently large odd cycle, Proposition~\ref{prop:powers-of-cycles} and Theorem~\ref{prop:operations} yield the following.

\begin{corollary}\label{cor:co-linear-forest} The complement of every disjoint union of paths is $1$-perfectly orientable.\end{corollary}

The following theorem describes a set of minimal forbidden induced minors for the class of \po~graphs.

\begin{sloppypar}
\begin{theorem}\label{thm:non-1-po} Let ${\cal F} = \{F_1,F_2,F_5,\ldots,F_{12}\}\cup {\cal F}_3\cup {\cal F}_4$, where: \begin{itemize}
   \item graphs $F_1$, $F_2$ are depicted in Fig.~\ref{fig:1}, and
  \item ${\cal F}_3 = \{\overline{C_{2k}}\mid k\ge 3\}$, the set of complements of even cycles of length at least $6$,
  \item ${\cal F}_4 = \{\overline{K_2+C_{2k+1}}\mid k\ge 1\}$, the set of complements of the graphs obtained as the disjoint union of $K_2$ and an
      odd cycle,
  \item for $i \in \{5,\ldots, 12\}$, graph $F_i$ is the complement of the graph $G_{i-4}$, depicted in Fig.~\ref{fig:1}.
\end{itemize} Then, every graph in set ${\cal F}$ is a minimal forbidden induced minor for the class of $1$-perfectly \hbox{orientable} graphs. \end{theorem}
\end{sloppypar}

 \begin{figure}[h!]
  \centering
\includegraphics[width=0.9\textwidth]{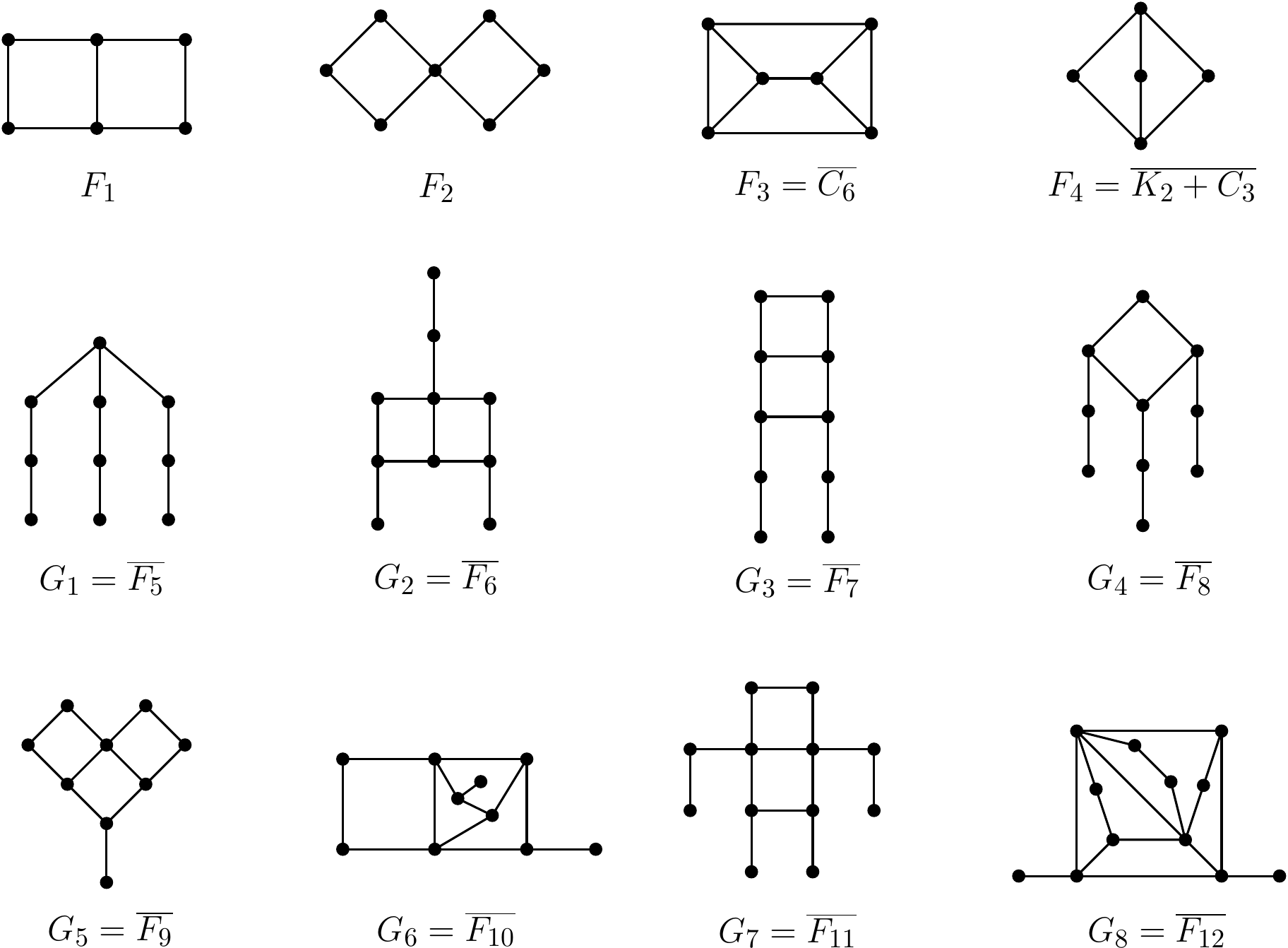} \caption{Four non-\po~graphs and $8$ complements of non-\po~graphs. Graphs $F_3$ and $F_4$ are the smallest members of families ${\cal F}_3$ and ${\cal F}_4$, respectively.}\label{fig:1} \end{figure}

\begin{proof} We need to show that each $F\in {\cal F}$ is not \po, but every proper induced minor of $F$ is. We first show that no graph in ${\cal F}$ is~$1$-p.o., and will argue minimality for all $F\in {\cal F}$ in the second part of the proof.

\medskip {\it No graph is in ${\cal F}$ is~$1$-p.o.}

First consider the graphs $F_1$ and $F_2$. Since they are both triangle-free, every edge clique cover of $F_i$ (for $i\in \{1,2\}$) contains all edges of $F_i$ and hence has at least $|E(F_i)|>|V(F_i)|$ members. Hence, Corollary~\ref{cor:upper-bound-on-theta} implies that $F_1$ and $F_2$ are not \po

The family ${\cal F}_3$ consists precisely of complements of even cycles of length at least $6$. In particular, every $F\in  {\cal F}_3$ is co-bipartite. By Theorem~\ref{thm:co-bipartite-circular-arc}, $F$ is \po~if and only if $F$ is circular arc. Since the family ${\cal F}_3$ is one of the six infinite families of minimal forbidden induced subgraphs for the class of circular arc co-bipartite graphs~\cite{MR0450140}, we infer that $F$ is not $1$-p.o.

Now let $F\in  {\cal F}_4$, that is, $F = \overline{K_2+C_{2k+1}}$ for some $k\ge 1$. We will prove that $F$ is not $1$-p.o.~using Lemma~\ref{lem:cycles}. Let the vertices of the cycle component of $\overline{F}$ be named $u_1,\ldots, u_{2k+1}$, according to a cyclic order of $C_{2k+1}$. Also, let the two vertices of the $K_2$ component of $\overline{F}$ be named $v_1$ and $v_2$. Suppose that $F$ admits a $1$-perfect orientation $D$. For every two consecutive vertices from the cycle we have an induced $C_4$ given by these two vertices together with $v_1$ and $v_2$. By Lemma~\ref{lem:cycles}, every such $C_4$ must be oriented cyclically. Consider the $C_4$ induced by vertices $\{u_1, u_2, v_1, v_2\}$. Without loss of generality we may assume that it is oriented as $v_1 \rightarrow{u_1} \rightarrow{v_2} \rightarrow{u_2} \rightarrow{v_1}$. This determines the orientation of the $C_4$ induced by $\{u_2,u_3,v_1,v_2\}$. Since the edge $\{v_{1},u_2\}$ is oriented as $u_2 \rightarrow{v_1}$, the edge $\{v_{1},u_3\}$ must be oriented as $v_1 \rightarrow{u_3}$. Proceeding along the cycle, we infer that $v_1\rightarrow u_i$ for odd $i$ and 
$u_i\rightarrow v_1$ for even $i$. However, this implies that $v_1\rightarrow u_1$ and 
$v_1\rightarrow u_{2k+1}$, contrary to the fact that $D$ is a $1$-perfect orientation of $F$. 
Therefore, $F$ is not \po

Each of the remaining graphs, $F_5$--$F_{12}$, belongs to the list of minimal forbidden induced subgraphs for the class of circular arc co-bipartite graphs~\cite{MR0450140}. By Theorem~\ref{thm:co-bipartite-circular-arc}, none of these graphs is $1$-p.o.

\medskip It remains to show minimality, that is, that {\it every proper induced minor of every graph in ${\cal F}$ is~$1$-p.o.}

First consider the graphs $F_1$ and $F_2$. Deleting any vertex of either $F_1$ or $F_2$ results in either a chordal graph or in a unicyclic graph, hence in a \po~graph (cf.~Proposition~\ref{prop:unicyclic}). Contracting any edge of $F_1$ results in a graph that is either chordal, or is obtained from a cycle by adding to it a simplicial vertex, hence in either case a~\po~graph. Contracting any edge of $F_2$ results in a graph that can be reduced to a cycle by removing true twins and simplicial vertices, hence this graph is also~\po

We are left with graphs that are defined in terms of their complements. To argue minimality for them, it will be convenient to understand the effect of performing the operation of edge contraction on a given graph on its complement. It can be seen that if $G$ is the graph obtained from a graph $H$ by contracting an edge $uv$, then $\overline{G}$ is the graph obtained from $\overline{H}$ identifying a pair of non-adjacent vertices (namely, $u$ and $v$) and making the new vertex adjacent exactly to the common neighbors in $\overline{H}$ of $u$ and $v$. We will refer to this operation as {\it co-contracting a non-edge}.

Let $F\in  {\cal F}_3$, that is, $F = \overline{C_{2k}}$ for some $k\ge 3$. Deleting a vertex from $F$ results in the complement of a path, which is~\po~by Corollary~\ref{cor:co-linear-forest}. Similarly, one can verify that co-contracting a non-edge of $\overline{F}$ results in a disjoint union of paths. Thus, every proper induced minor of $F$ is \po

Let $F\in  {\cal F}_4$, that is, $F = \overline{C_{2k+1}+K_2}$ for some $k\ge 1$. Deleting a vertex in the cycle component of $\overline{F}$ from $F$ results in the complement of a disjoint union of path, which is~\po~by Corollary~\ref{cor:co-linear-forest}. Deleting a vertex in the $K_2$ component of $\overline{F}$ from $F$ results in the graph that consists of the join of $K_1$ and the complement of an odd cycle, which is \po~by Theorem~\ref{prop:operations} and~Proposition~\ref{prop:powers-of-cycles}. Furthermore, co-contracting a non-edge of $\overline{F}$ results in a disjoint union of paths, and Corollary~\ref{cor:co-linear-forest} applies again. Thus, every proper induced minor of $F$ is \po{}

We recall that each of the remaining graphs, $F_5$--$F_{12}$, is a minimal forbidden induced subgraph for the class of circular arc co-bipartite graphs. Deleting a vertex from any of them results in a circular arc co-bipartite graph, hence in a \po~graph. Note that $F_{9}$ has $9$ vertices, each of $F_5, F_6, F_7, F_8, F_{10}$ has $10$ vertices, and $F_{11}$ and $F_{12}$ have $12$ vertices. Also note that since co-bipartite graphs are closed under edge contractions, in order to show that every graph obtained from one of the graphs $F_5$--$F_{12}$ by contracting an edge is \po, it suffices to argue that it is circular arc, which (since it is co-bipartite) is equivalent to verifying that it does not contain any of the minimal forbidden induced subgraphs for the class of circular arc co-bipartite graphs~\cite{MR0450140}. The only graphs with at most $10$ vertices on this list are $\overline{C_6}$, $\overline{C_8}$, $\overline{C_{10}}$, and graphs $F_5$--$F_{10}$. The list also contains a unique graph of order $11$; let $G_9$ denote its complement. Let $G\in \{\overline{F_5}, \ldots, \overline{F_{12}}\} = \{G_1,\ldots, G_8\}$. A direct inspection of the possible graphs resulting from co-contracting a non-edge of $G$ shows that every such graph has either at most $10$ vertices, in which case its complement is $\{C_6,C_8,C_{10}, G_1,\ldots, G_6\}$-free, or it has $11$ vertices, in which case its complement either has an isolated vertex and the rest is $\{C_6,C_8,C_{10}, G_1,\ldots, G_6\}$-free, or it is connected, of order $11$, and $\{C_6,C_8,C_{10}, G_1,\ldots, G_6,G_9\}$-free. Thus, in all cases contracting an edge of a graph in $\{F_5, \ldots, F_{12}\}$ results in a circular arc graph, hence in a \po~graph. This completes the proof. \end{proof}

Theorem~\ref{thm:non-1-po} implies that ${\cal F} \subseteq \tilde{\cal F}$, where $\tilde{\cal F}$ is the set of minimal forbidden induced minors for the class of \po~graphs. However, the complete set $\tilde{\cal F}$ is unknown. It is conceivable that one can obtain further graphs in $\tilde{\cal F}$ by computing the minimal elements with respect to the induced minor relation of the list of forbidden induced subgraphs for the class of circular arc co-bipartite graphs due to Trotter and Moore~\cite{MR0450140}. Besides the three small graphs $F_5,F_6,F_7$ and the family ${\cal F}_3$ of complements of even cycles of length at least $6$, the list contains five other infinite families, the smallest members of which are graphs $F_8,\ldots, F_{12}$, respectively. (In~\cite{MR0450140}, the lists represent the complementary property and are denoted by ${\cal T}_i$, ${\cal W}_i$, ${\cal D}_i$, ${\cal M}_i$, and ${\cal N}_i$, respectively.)

\begin{sloppypar}
\begin{problem} Determine the set of minimal forbidden induced minors for the class of {$1$-perfectly orientable graphs}. \end{problem}
\end{sloppypar}

\section{$1$-perfectly orientable cographs}\label{sec:cographs}

Bang-Jensen et al.~\cite{MR1244934} characterized \po~line graphs and \hbox{\po}~triangle-free graphs, and in Section~\ref{sec:co-bipartite-graphs} we characterized \po~co-bipartite graphs. We conclude the paper by characterizing \po~cographs, obtaining characterizations both in terms of forbidden induced subgraphs and in terms of structural properties. The class of {\it cographs} is defined recursively by stating that $K_1$ is a cograph, the disjoint union of two cographs is a cograph, the join of two cographs is a cograph, and there are no other cographs. It is well known (see, e.g.,~\cite{MR1686154}) that cographs can be characterized in terms of forbidden induced subgraphs by a single obstruction, namely the $4$-vertex path $P_4$.

\begin{theorem}\label{thm:cograph} For every cograph $G$, the following conditions are equivalent: \begin{enumerate}
  \item $G$ is $1$-perfectly orientable.
  \item $G$ is $K_{2,3}$-free.
  \item One of the following conditions holds:
\begin{itemize}
  \item $G\cong K_1$.
  \item $G\cong \overline{mK_2}$ for some $m\ge 2$.
  \item $G$ is the disjoint union of two smaller \po~cographs.
  \item $G$ is obtained from a \po~cograph by adding to it a universal vertex.
  \item $G$ is obtained from a \po~cograph by adding to it a true twin.
  \end{itemize}
\end{enumerate} \end{theorem}

\begin{proof} The implication $1\Rightarrow 2$ follows from Theorems~\ref{prop:operations} and~\ref{thm:non-1-po}.

To show the implication $2\Rightarrow 3$, suppose that $G$ is a $K_{2,3}$-free cograph on at least two vertices that is not disconnected and does not have a universal vertex or a pair of true twins. We want to show that $G=\overline{mK_2}$. Since $G$ is not disconnected and $G\neq K_1$, its complement $\overline{G}$ is disconnected. Let $m\ge 2$ denote the number of {\it co-components} of $G$ (subgraphs of $G$ induced by the vertex sets of components of $\overline{G}$). If one of the co-components has exactly one vertex, then that vertex is universal in $G$, which is a contradiction. Therefore, each co-components has at least two vertices. The recursive structure of cographs implies that each co-component of $G$ is disconnected. In particular, it has independence number at least $2$. On the other hand, since $G$ is $K_{2,3}$-free, each co-component of $G$ has independence number at most $2$. This implies that each co-component is the disjoint union of two complete graphs. Since $G$ has no true twins, each co-component is isomorphic to $2K_1$, that is, $G\cong \overline{mK_2}$ for some $m\ge 2$, as claimed.

Finally, we show the implication $3\Rightarrow 1$. Suppose that $G$ is a cograph such that one of the five conditions in item $3$ holds. An inductive argument shows that $G$ is \po, using Theorem~\ref{prop:operations} and the fact that $K_1$ and all graphs of the form $\overline{mK_2}$ are \po~(which follows, e.g.,~from Corollary~\ref{cor:co-linear-forest}).
\end{proof}

\subsection*{Acknowledgement}

We are grateful to Aistis Atminas, Marcin J.~Kami\'nski, and Nicolas Trotignon for useful discussions on the topic and to Pavol Hell for helpful comments on an earlier draft.

\bibliographystyle{abbrv}

\begin{sloppypar} \bibliography{1-perfectly-orientable-bib}{} \end{sloppypar}

\def\soft#1{\leavevmode\setbox0=\hbox{h}\dimen7=\ht0\advance \dimen7
  by-1ex\relax\if t#1\relax\rlap{\raise.6\dimen7
  \hbox{\kern.3ex\char'47}}#1\relax\else\if T#1\relax
  \rlap{\raise.5\dimen7\hbox{\kern1.3ex\char'47}}#1\relax \else\if
  d#1\relax\rlap{\raise.5\dimen7\hbox{\kern.9ex \char'47}}#1\relax\else\if
  D#1\relax\rlap{\raise.5\dimen7 \hbox{\kern1.4ex\char'47}}#1\relax\else\if
  l#1\relax \rlap{\raise.5\dimen7\hbox{\kern.4ex\char'47}}#1\relax \else\if
  L#1\relax\rlap{\raise.5\dimen7\hbox{\kern.7ex
  \char'47}}#1\relax\else\message{accent \string\soft \space #1 not
  defined!}#1\relax\fi\fi\fi\fi\fi\fi} \def\cprime{$'$}
\begin{thebibliography}{10}

\bibitem{MR2472389}
J.~Bang-Jensen and G.~Gutin.
\newblock {\em Digraphs}.
\newblock Springer Monographs in Mathematics. Springer-Verlag London, Ltd.,
  London, second edition, 2009.
\newblock Theory, algorithms and applications.

\bibitem{MR1244934}
J.~Bang-Jensen, J.~Huang, and E.~Prisner.
\newblock In-tournament digraphs.
\newblock {\em J. Combin. Theory Ser. B}, 59(2):267--287, 1993.

\bibitem{MR1686154}
A.~Brandst{\"a}dt, V.~B. Le, and J.~P. Spinrad.
\newblock {\em Graph {C}lasses: {A} {S}urvey}.
\newblock SIAM Monographs on Discrete Mathematics and Applications. Society for
  Industrial and Applied Mathematics (SIAM), Philadelphia, PA, 1999.

\bibitem{MR2901082}
F.~Cicalese and M.~Milani{\v{c}}.
\newblock Graphs of separability at most 2.
\newblock {\em Discrete Appl. Math.}, 160(6):685--696, 2012.

\bibitem{MR2744811}
R.~Diestel.
\newblock {\em Graph {T}heory}, volume 173 of {\em Graduate Texts in
  Mathematics}.
\newblock Springer, Heidelberg, fourth edition, 2010.

\bibitem{MR3159129}
G.~Dur{\'a}n, L.~N. Grippo, and M.~D. Safe.
\newblock Structural results on circular-arc graphs and circle graphs: {A}
  survey and the main open problems.
\newblock {\em Discrete Appl. Math.}, 164(part 2):427--443, 2014.

\bibitem{MR712930}
R.~D. Dutton and R.~C. Brigham.
\newblock A characterization of competition graphs.
\newblock {\em Discrete Appl. Math.}, 6(3):315--317, 1983.

\bibitem{MR1246675}
H.~Galeana-S{\'a}nchez.
\newblock Normal fraternally orientable graphs satisfy the strong perfect graph
  conjecture.
\newblock {\em Discrete Math.}, 122(1-3):167--177, 1993.

\bibitem{MR1449722}
H.~Galeana-S{\'a}nchez.
\newblock A characterization of normal fraternally orientable perfect graphs.
\newblock {\em Discrete Math.}, 169(1-3):221--225, 1997.

\bibitem{MR1292980}
F.~Gavril.
\newblock Intersection graphs of proper subtrees of unicyclic graphs.
\newblock {\em J. Graph Theory}, 18(6):615--627, 1994.

\bibitem{MR1287025}
F.~Gavril and J.~Urrutia.
\newblock Intersection graphs of concatenable subtrees of graphs.
\newblock {\em Discrete Appl. Math.}, 52(2):195--209, 1994.

\bibitem{MR2063679}
M.~C. Golumbic.
\newblock {\em Algorithmic {G}raph {T}heory and {P}perfect {G}raphs}, volume~57
  of {\em Annals of Discrete Mathematics}.
\newblock Elsevier Science B.V., Amsterdam, second edition, 2004.
\newblock With a foreword by Claude Berge.

\bibitem{MR2370526}
V.~Gurvich.
\newblock On cyclically orientable graphs.
\newblock {\em Discrete Math.}, 308(1):129--135, 2008.

\bibitem{MR1081957}
P.~Hell, J.~Bang-Jensen, and J.~Huang.
\newblock Local tournaments and proper circular arc graphs.
\newblock In {\em Algorithms ({T}okyo, 1990)}, volume 450 of {\em Lecture Notes
  in Comput. Sci.}, pages 101--108. Springer, Berlin, 1990.

\bibitem{MR1435657}
P.~Hell and J.~Huang.
\newblock Two remarks on circular arc graphs.
\newblock {\em Graphs Combin.}, 13(1):65--72, 1997.

\bibitem{MR1415290}
C.~R. Johnson and T.~A. McKee.
\newblock Structural conditions for cycle completable graphs.
\newblock {\em Discrete Math.}, 159(1-3):155--160, 1996.

\bibitem{MR3152051}
F.~Kammer and T.~Tholey.
\newblock Approximation algorithms for intersection graphs.
\newblock {\em Algorithmica}, 68(2):312--336, 2014.

\bibitem{MR1109419}
J.~Kratochv{\'{\i}}l.
\newblock String graphs. {I}. {T}he number of critical nonstring graphs is
  infinite.
\newblock {\em J. Combin. Theory Ser. B}, 52(1):53--66, 1991.

\bibitem{MR1360111}
H.-J. Lai.
\newblock Super-{E}ulerian graphs and excluded induced minors.
\newblock {\em Discrete Math.}, 146(1-3):133--143, 1995.

\bibitem{MR2567965}
M.~C. Lin and J.~L. Szwarcfiter.
\newblock Characterizations and recognition of circular-arc graphs and
  subclasses: a survey.
\newblock {\em Discrete Math.}, 309(18):5618--5635, 2009.

\bibitem{MR1672910}
T.~A. McKee and F.~R. McMorris.
\newblock {\em Topics in {I}ntersection {G}raph {T}heory}.
\newblock SIAM Monographs on Discrete Mathematics and Applications. Society for
  Industrial and Applied Mathematics (SIAM), Philadelphia, PA, 1999.

\bibitem{MR1293386}
R.~Motwani and M.~Sudan.
\newblock Computing roots of graphs is hard.
\newblock {\em Discrete Appl. Math.}, 54(1):81--88, 1994.

\bibitem{MR0210616}
A.~Mukhopadhyay.
\newblock The square root of a graph.
\newblock {\em J. Combinatorial Theory}, 2:290--295, 1967.

\bibitem{MR2323998}
J.~Ne{\v{s}}et{\v{r}}il and P.~Ossona~de Mendez.
\newblock Fraternal augmentations of graphs, coloration and minors.
\newblock In {\em 6th {C}zech-{S}lovak {I}nternational {S}ymposium on
  {C}ombinatorics, {G}raph {T}heory, {A}lgorithms and {A}pplications},
  volume~28 of {\em Electron. Notes Discrete Math.}, pages 223--230.

\bibitem{MR2548660}
J.~Ne{\v{s}}et{\v{r}}il and P.~Ossona~de Mendez.
\newblock Fraternal augmentations, arrangeability and linear {R}amsey numbers.
\newblock {\em European J. Combin.}, 30(7):1696--1703, 2009.

\bibitem{MR679638}
R.~J. Opsut.
\newblock On the computation of the competition number of a graph.
\newblock {\em SIAM J. Algebraic Discrete Methods}, 3(4):420--428, 1982.

\bibitem{prisner1988familien}
E.~Prisner.
\newblock {\em Familien zusammenh{\"a}ngender Teilgraphen eines Graphen und
  ihre Durchschnittsgraphen}.
\newblock Dissertation Hamburg, 1988.

\bibitem{MR1206558}
A.~Raychaudhuri.
\newblock On powers of strongly chordal and circular arc graphs.
\newblock {\em Ars Combin.}, 34:147--160, 1992.

\bibitem{MR666799}
D.~J. Skrien.
\newblock A relationship between triangulated graphs, comparability graphs,
  proper interval graphs, proper circular-arc graphs, and nested interval
  graphs.
\newblock {\em J. Graph Theory}, 6(3):309--316, 1982.

\bibitem{MR0450140}
W.~T. Trotter, Jr. and J.~I. Moore, Jr.
\newblock Characterization problems for graphs, partially ordered sets,
  lattices, and families of sets.
\newblock {\em Discrete Math.}, 16(4):361--381, 1976.

\bibitem{MR0276129}
A.~Tucker.
\newblock Characterizing circular-arc graphs.
\newblock {\em Bull. Amer. Math. Soc.}, 76:1257--1260, 1970.

\bibitem{MR1161986}
J.~Urrutia and F.~Gavril.
\newblock An algorithm for fraternal orientation of graphs.
\newblock {\em Inform. Process. Lett.}, 41(5):271--274, 1992.

\bibitem{MR2643278}
T.~Zou.
\newblock Cyclical orientations of graphs.
\newblock {\em Sichuan Daxue Xuebao}, 47(1):21--26, 2010.

\end{thebibliography}

\end{document}